\documentclass[12pt, reqno]{jams-l}

\usepackage{amsmath,enumerate,amsfonts,amssymb,color,graphicx,amsthm}

\usepackage[colorlinks=true, allcolors=black]{hyperref}
\usepackage{todonotes}
\usepackage[normalem]{ulem}
\numberwithin{equation}{section}
\usepackage{cite}
\usepackage{mathtools}
\usepackage[toc,page]{appendix}
\usepackage[margin=3cm]{geometry}

\newlength{\leftstackrelawd}
\newlength{\leftstackrelbwd}
\def\leftstackrel#1#2{\settowidth{\leftstackrelawd}%
	{${{}^{#1}}$}\settowidth{\leftstackrelbwd}{$#2$}%
	\addtolength{\leftstackrelawd}{-\leftstackrelbwd}%
	\leavevmode\ifthenelse{\lengthtest{\leftstackrelawd>0pt}}%
	{\kern-.5\leftstackrelawd}{}\mathrel{\mathop{#2}\limits^{#1}}}

\theoremstyle{plain}
\newtheorem{thm}{Theorem}[section]
\newtheorem{lem}[thm]{Lemma}
\newtheorem{cor}[thm]{Corollary}
\newtheorem{prop}[thm]{Proposition}
\newtheorem*{thm*}{Theorem}

\theoremstyle{definition}

\newtheorem{rmk}[thm]{Remark}
\newtheorem{?}[thm]{Problem}

\newcommand{\ep}{\varepsilon}
\renewcommand{\phi}{\varphi}
\renewcommand{\epsilon}{\varepsilon}
\makeatletter
\def\@cite#1#2{[\textbf{#1\if@tempswa , #2\fi}]}
\def\@biblabel#1{[\textbf{#1}]}
\makeatother

\makeatletter
\newcommand*{\defeq}{\mathrel{\rlap{%
			\raisebox{0.3ex}{$\m@th\cdot$}}%
		\raisebox{-0.3ex}{$\m@th\cdot$}}%
	=}
\makeatother

\makeatletter
\newcommand*{\eqdef}{=\mathrel{\rlap{%
			\raisebox{0.3ex}{$\m@th\cdot$}}%
		\raisebox{-0.3ex}{$\m@th\cdot$}}%
}
\makeatother

\newcounter{marnote}

\makeatletter
\def\underbracex#1#2{\mathop{\vtop{\m@th\ialign{##\crcr
				$\hfil\displaystyle{#2}\hfil$\crcr
				\noalign{\kern3\p@\nointerlineskip}%
				#1\crcr\noalign{\kern3\p@}}}}\limits}

\def\upbracefilla{$\m@th \setbox\z@\hbox{$\braceld$}%
	\bracelu\leaders\vrule \@height\ht\z@ \@depth\z@\hfill 
	\kern\p@\vrule \@width\p@\kern\p@\vrule \@width\p@\kern\p@\vrule \@width\p@
	$}

\def\upbracefillb{$\m@th \setbox\z@\hbox{$\braceld$}%
	\vrule \@width\p@\kern\p@\vrule \@width\p@\kern\p@\vrule \@width\p@\kern\p@
	\leaders\vrule \@height\ht\z@ \@depth\z@\hfill\bracerd
	\braceld\leaders\vrule \@height\ht\z@ \@depth\z@\hfill
	\kern\p@\vrule \@width\p@\kern\p@\vrule \@width\p@\kern\p@\vrule \@width\p@
	$}

\def\upbracefillc{$\m@th \setbox\z@\hbox{$\braceld$}%
	\vrule \@width\p@\kern\p@\vrule \@width\p@\kern\p@\vrule \@width\p@\kern\p@
	\leaders\vrule \@height\ht\z@ \@depth\z@\hfill
	\kern\p@\vrule \@width\p@\kern\p@\vrule \@width\p@\kern\p@\vrule \@width\p@
	$}

\def\upbracefilld{$\m@th \setbox\z@\hbox{$\braceld$}%
	\vrule \@width\p@\kern\p@\vrule \@width\p@\kern\p@\vrule \@width\p@\kern\p@
	\leaders\vrule \@height\ht\z@ \@depth\z@\hfill\braceru$}

\def\upbracefillbd{$\m@th \setbox\z@\hbox{$\braceld$}%
	\vrule \@width\p@\kern\p@\vrule \@width\p@\kern\p@\vrule \@width\p@\kern\p@
	\bracerd\braceld
	\leaders\vrule \@height\ht\z@ \@depth\z@\hfill\braceru$}

\makeatother

\makeatletter
\def\l@subsection{\@tocline{2}{0pt}{2.5pc}{5pc}{}}
\makeatother

\begin{document}
	\title[An eigenvalue estimate for minimal hypersurfaces in the sphere]{An improved eigenvalue estimate for embedded minimal hypersurfaces in the sphere}
	\author{Jonah A. J. Duncan}
	\address{Johns Hopkins University, 404 Krieger Hall, Department of Mathematics, 3400 N. Charles Street, Baltimore, MD 21218, US.}
	\curraddr{}
	\email{jdunca33@jhu.edu}
	
	\author{Yannick Sire}
	\address{Johns Hopkins University, 404 Krieger Hall, Department of Mathematics, 3400 N. Charles Street, Baltimore, MD 21218, US.}
	\curraddr{}
	\email{ysire1@jhu.edu}
	
	\author{Joel Spruck}
	\address{Johns Hopkins University, 404 Krieger Hall, Department of Mathematics, 3400 N. Charles Street, Baltimore, MD 21218, US.}
	\curraddr{}
	\email{jspruck1@jhu.edu}
	
	\maketitle
	
	\vspace*{-1mm}

	\begin{abstract}
	Suppose that $\Sigma^n\subset\mathbb{S}^{n+1}$ is a closed embedded minimal hypersurface. We prove that the first non-zero eigenvalue $\lambda_1$ of the induced Laplace-Beltrami operator on $\Sigma$ satisfies $\lambda_1 \geq \frac{n}{2}+ a_n(\Lambda^6 + b_n)^{-1}$, where $a_n$ and $b_n$ are explicit dimensional constants and $\Lambda$ is an upper bound for the length of the second fundamental form of $\Sigma$. This provides the first explicitly computable improvement on Choi \& Wang's lower bound $\lambda_1 \geq \frac{n}{2}$ without any further assumptions on $\Sigma$.
\end{abstract}

\section{Introduction}

An important problem in geometric analysis is to understand the spectrum of the Laplace-Beltrami operator on a Riemannian manifold, and to study its relation to the underlying intrinsic and/or ambient geometry. From the geometric perspective, it is of particular interest to address such questions for manifolds embedded in spaces of constant curvature. In this paper, we obtain a new lower bound for the first non-zero eigenvalue $\lambda_1(\Sigma)$ of the induced Laplace-Beltrami operator $-\Delta^\Sigma$ on a smooth closed hypersurface $\Sigma^n$ minimally embedded in the unit sphere $\mathbb{S}^{n+1}$ (which we always assume to be equipped with the round metric). \medskip 

In this direction, an argument of Choi \& Wang \cite{CW83}, later refined by Brendle \cite{Bre13}, gives the lower bound
\begin{align}\label{1}
\lambda_1(\Sigma) > \frac{n}{2}.
\end{align}
An important application of \eqref{1} and the Yang-Yau inequality \cite{YY80} is an area bound for embedded minimal surfaces in $\mathbb{S}^3$ in terms of their genus; this plays a crucial role in the compactness theory of Choi \& Schoen \cite{CS85}. Moreover, \eqref{1} provides evidence towards a famous conjecture of Yau \cite{Yau82}, which predicts that $\lambda_1(\Sigma)$ is equal to $n$. Note that the restriction to $\mathbb{S}^{n+1}$ of each coordinate function on $\mathbb{R}^{n+2}$ is an eigenfunction for $-\Delta^\Sigma$ with eigenvalue $n$, and thus the upper bound $\lambda_1(\Sigma) \leq n$ is clear. \medskip 

Despite an extensive literature relating to the study of $\lambda_1(\Sigma)$ under additional assumptions on $\Sigma$ since the work of Choi \& Wang (see e.g.~\cite{CS09, TY13} and the references therein), \eqref{1} has remained the strongest \textit{explicit} lower bound that is known to hold for a general embedded minimal hypersurface in $\mathbb{S}^{n+1}$. In this paper, we obtain an explicit improvement on \eqref{1} which depends only on the dimension $n$ and an upper bound on the length $\|A\| \defeq \sqrt{\operatorname{trace}A^2}$ of the second fundamental form of $\Sigma$. Our main result is as follows:

\begin{thm}\label{A}
	Let $\Sigma^n\subset\mathbb{S}^{n+1}$ be a closed embedded minimal hypersurface and denote $\Lambda = \max_{\Sigma}\|A\|$. Then there exist constants 
	\begin{align}\label{50'}
	a_n \geq  \frac{(n-1)n^{2}}{32000} \quad \text{and} \quad b_n \leq \frac{5n^2}{216}
	\end{align}
	such that
	\begin{align}\label{-1}
	\lambda_1(\Sigma) \geq \frac{n}{2} + \frac{a_n}{\Lambda^6 + b_n}. 
	\end{align}
\end{thm}

\begin{rmk}\label{53}
	In the proof of Theorem \ref{A}, we will actually show that one can take 
		\begin{align}\label{50}
	a_n \geq  \frac{3(n-1)n^{7/2}}{3200}\arctan^3\bigg(\frac{1}{3\sqrt{n}}\bigg) \quad \text{and} \quad b_n \leq \frac{5n^{7/2}}{8}\arctan^3\bigg(\frac{1}{3\sqrt{n}}\bigg). 
	\end{align}
	The inequalities in \eqref{50'} follow since for $n\geq 2$ we have $\frac{7}{200} \leq n^{3/2}\arctan^3(\frac{1}{3\sqrt{n}})\leq \frac{1}{27}$. 
\end{rmk}

Whilst we are only interested in explicitly computable lower bounds for $\lambda_1(\Sigma)$ in this paper, we note that upper bounds for either $\lambda_1(M^n,g)$ or $\lambda_1(M^n,g)\cdot \operatorname{Vol}(M^n,g)^{2/n}$ on Riemannian manifolds $(M^n,g)$ have also been studied extensively -- see for instance the classical works of Cheng \cite{Cheng75}, Li \& Yau \cite{LY80, LY82}, Yang \& Yau \cite{YY80} and Korevaar \cite{Kor93}. In particular, recall that for a closed orientable Riemannian surface $(\Sigma^2, g)$ of genus $\gamma$, the Yang-Yau inequality \cite{YY80, EI84} states that $\lambda_1(\Sigma,g)\operatorname{Area}(\Sigma,g) \leq 8\pi \lfloor \frac{\gamma+3}{2}\rfloor$, where $\lfloor x \rfloor$ denotes the integer part of $x$. The following result is then an immediate corollary of the Yang-Yau inequality and Theorem \ref{A}:

\begin{cor}
	Let $\Sigma^2\subset\mathbb{S}^{3}$ be a closed embedded minimal surface of genus $\gamma$ and denote $\Lambda = \max_{\Sigma}\|A\|$. Then there exist constants $a_n$ and $b_n$ satisfying \eqref{50'} such that 
	\begin{align}
	\operatorname{Area}(\Sigma) \leq \bigg(\frac{n}{2} + \frac{a_n}{\Lambda^6 + b_n}\bigg)^{-1}8\pi \bigg\lfloor \frac{\gamma+3}{2}\bigg\rfloor. 
	\end{align}
\end{cor}

\begin{rmk}
	As a consequence of our \textit{method} for proving Theorem \ref{A}, we will also obtain an explicit volume bound for closed embedded mean-convex hypersurfaces in $\mathbb{S}^{n+1}$ in terms of $n$ and $\Lambda$ -- see Proposition \ref{C}. We note that our proof of Proposition \ref{C} does not invoke any lower bound for $\lambda_1$. 
\end{rmk}

To put Theorem \ref{A} into context, we now briefly discuss some related results. We first observe that, in light of the strictness of the inequality in \eqref{1}, non-explicit improved lower bounds depending only on quantities such as dimension, index, topological type and curvature bounds follow from suitable compactness results. For example, if $\mathcal{A}(\Lambda, n)$ denotes the class of closed embedded minimal hypersurfaces in $\mathbb{S}^{n+1}$ satisfying $\max_\Sigma\|A\| \leq \Lambda$, then it is well-known that $\mathcal{A}(\Lambda, n)$ is compact in the $C^k$ topology for any $k\geq 2$. Combined with \eqref{1}, it follows that there exists a constant $\alpha(\Lambda, n)>0$ such that 
\begin{align}\label{9}
\lambda_1(\Sigma) \geq \frac{n}{2} + \alpha(\Lambda,n)\quad\text{for all }\Sigma\in\mathcal{A}(\Lambda, n). 
\end{align}
We stress that, in contrast with \eqref{9}, the estimate \eqref{-1} obtained in Theorem \ref{A} provides an \textit{explicitly computable} improvement on \eqref{1}. Moreover, our lower bound \eqref{-1} is obtained by arguing more directly in the spirit of \cite{CW83}, rather than appealing to any compactness theory. \medskip

$C^k$ compactness results have also been established in other classes. For example, Choi \& Schoen showed in \cite{CS85} that the class $\mathcal{B}(\gamma)$ of closed embedded minimal surfaces in $\mathbb{S}^3$ with genus less than $\gamma$ is compact in the $C^k$ topology for any $k\geq 2$. In combination with \eqref{1}, this implies the existence of a constant $\beta(\gamma)>0$ such that
\begin{align}
\lambda_1(\Sigma) \geq 1 +\beta(\gamma) \quad \text{for all }\Sigma\in\mathcal{B}(\gamma). 
\end{align}
A more recent compactness of result of Sharp \cite[Corollary 2.6]{Sharp17} shows that the class $\mathcal{C}(V,I,n)$ of closed embedded minimal hypersurfaces in $\mathbb{S}^{n+1}$ with volume bounded from above by $V$ and index bounded from above by $I$ is compact in the $C^k$ topology for $k\geq 2$ when $2\leq n \leq 6$. Combined with \eqref{1}, this implies the existence of a constant $\delta(V,I,n)>0$ such that 
\begin{align}
\lambda_1(\Sigma) \geq \frac{n}{2}+\delta(V,I, n) \quad \text{for all }\Sigma\in\mathcal{C}(V,I,n) \text{ when }2 \leq n \leq 6. 
\end{align}

In a similar vein to Theorem \ref{A}, it would be interesting to derive improved lower bounds for $\lambda_1(\Sigma)$ with explicit dependence on quantities such as genus (when $n=2$), volume and/or index. Such results could provide a step towards proving Yau's conjecture within certain classes of minimal hypersurfaces in $\mathbb{S}^{n+1}$. Recently, Yau's conjecture was established for the class of embedded isoparametric minimal hypersurfaces in $\mathbb{S}^{n+1}$ -- see Tang \& Yan \cite{TY13} and the references therein. We refer also to the work of Choe \& Soret \cite{CS09}, where Yau's conjecture was established for a class of symmetric minimal surfaces in $\mathbb{S}^3$. 

\begin{rmk}
	The aforementioned results of \cite{CW83, CS85, Sharp17} apply more generally when $\mathbb{S}^{n+1}$ is replaced by a closed manifold $(M^{n+1},g)$ whose Ricci curvature satisfies $\operatorname{Ric}_g \geq kg$ for some constant $k>0$. The bound \eqref{1} is then replaced by $\lambda_1(\Sigma) > \frac{k}{2}$, and our subsequent discussion generalises in the obvious way. In attempting to generalise Theorem \ref{A} to this more general context, it seems that our method introduces constants that depend on \textit{sectional curvature bounds}. To keep the exposition simple, and since the case of the sphere is the one of most interest, we will not discuss such generalisations in this paper. 
\end{rmk}

The plan of the paper is as follows. In Section \ref{s2} we prove a preliminary result on the embeddedness of parallel hypersurfaces in $\mathbb{S}^{n+1}$. As a corollary, we obtain an explicit volume bound for closed embedded mean-convex hypersurfaces in $\mathbb{S}^{n+1}$ in terms of an upper bound for $\|A\|$. In Section \ref{s3} we prove Theorem \ref{A}. The key here is to estimate a positive term which is dropped in the estimate of Choi \& Wang in \cite{CW83}. Here, our integral estimates require working in a neighbourhood of $\Sigma$ whose thickness is controlled away from zero; this control is provided by our results in Section \ref{s2}.

\section{Embeddedness of parallel hypersurfaces}\label{s2}

Suppose that $\Sigma^n$ is a smooth, closed and embedded hypersurface in $\mathbb{S}^{n+1}$. As observed in \cite{CW83}, $\Sigma$ divides the sphere into two components $\mathbb{S}^{n+1}= M_1\cup M_2$, where $\partial M_1 = \partial M_2 = \Sigma$. Let $N\Sigma$ denote the normal bundle of $\Sigma \subset \mathbb{S}^{n+1}$ and $\operatorname{exp}^{N\Sigma}$ the corresponding exponential map. We fix the orientation on $\Sigma$ determined by the unit normal vector field $X$ on $\Sigma$ pointing into $M_1$, and for $t\in\mathbb{R}$ we define
\begin{align}
\Sigma^t = \{\operatorname{exp}^{N\Sigma}(p,tX_p)\in \mathbb{S}^{n+1}:p\in\Sigma\}. 
\end{align}
Geometrically, $\Sigma^t$ is the hypersurface parallel to $\Sigma$ and of signed distance $t$ to $\Sigma$. It is well-known (see e.g.~Theorems 2.1 and 2.2 in \cite{CR15}) that if $\kappa_1(p),\dots,\kappa_n(p)$ are the principal curvatures of $\Sigma$ at $p$ and $\kappa_{\operatorname{max}} = \max_{p\in\Sigma, i\in\{1,\dots,n\}}|\kappa_i(p)|$, then $\Sigma^t$ is a smooth immersed hypersurface in $\mathbb{S}^{n+1}$ for 
\begin{align}\label{35}
|t| < \arctan(\kappa_{\operatorname{max}}^{-1}) \eqdef T_\Sigma. 
\end{align}
Moreover, we may consider $n$ continuous functions $\kappa_i(\cdot, \cdot): \Sigma \times (-T_\Sigma, T_\Sigma)\rightarrow \mathbb{R}$ defined by 
\begin{align}\label{36}
\kappa_i(\cdot, t) = \frac{\kappa_i(\cdot ,0 ) + \tan t}{1-\kappa_i(\cdot, 0)\tan t}, \quad \kappa_i(\cdot,0) \defeq \kappa_i(\cdot). 
\end{align}
Then for each $t\in(-T_\Sigma, T_\Sigma)$, the quantities $\kappa_1(p,t),\dots,\kappa_n(p,t)$ are the principal curvatures of $\Sigma^t$ at $\operatorname{exp}^{N\Sigma}(p,tX_p)$, with respect to the orientation on $\Sigma^t$ determined by parallel transporting $X$ along geodesics normal to $\Sigma$ by a signed distance $t$. \medskip

Whilst it is well-known that $\Sigma^t$ remains embedded for $t$ sufficiently small, in general the range of $t$ for which $\Sigma^t$ is embedded is \textit{not} controlled by the curvature of $\Sigma$, since $\Sigma$ may be arbitrarily close to `self-touching'. We show that in the case that $\Sigma$ is mean-convex (that is, the mean curvature $H_\Sigma$ of $\Sigma$ is nonnegative), we do in fact have such control:

\begin{prop}\label{B}
	Suppose $\Sigma^n\subset\mathbb{S}^{n+1}$ is a smooth, closed and embedded mean-convex hypersurface. Then $\Sigma^t$ is a smooth, closed and embedded strictly mean-convex hypersurface in $\mathbb{S}^{n+1}$ for $|t|\in(0,T_\Sigma)$. 
\end{prop}
\begin{proof}
	We consider the case $t> 0$; the case $t< 0$ is similar. Let
	\begin{align}
	t_* = \sup \{t>0: \Sigma^t \text{ is smooth and embedded}\},
	\end{align}
	and suppose for a contradiction that $ t_* = \arctan(\ep\kappa_{\operatorname{max}}^{-1})$ for some $0<\ep<1$. Since $t_*<T_\Sigma$, by the discussion above $\Sigma^{t_*}$ is a smooth, immersed, non-embedded hypersurface. Therefore, for some point $x\in \Sigma^{t_*}$, there exist distinct points $p,q\in\Sigma$ such that $x = \operatorname{exp}^{N\Sigma}(p, t_*X_p) = \operatorname{exp}^{N\Sigma}(q, t_* X_q)$. Now, locally near $p$ (resp.~$q$), $\Sigma^t$ is a smooth graph over a neighbourhood of the origin in $T_p \Sigma$ (resp.~$T_q \Sigma$) for $t\leq t_*$. Denote these graphs by $\Sigma^t_p$ and $\Sigma^t_q$, respectively. Then by \eqref{36}, on $\Sigma_p^t$ we have
	\begin{align}\label{30}
	\kappa_{i}(p,t_*) = \frac{\kappa_i(p,0) +  \tan  t_*}{1-\kappa_i(p,0)\tan t_*} & = \kappa_i(p,0) + \frac{(1+\kappa_i(p,0)^2)\tan  t_*}{1-\kappa_i(p,0)\tan t_*}  \nonumber \\
	& > \kappa_i(p,0) + \frac{(1+\kappa_i(p,0)^2)\tan  t_*}{1-\ep} ,
	\end{align}
	and likewise on $\Sigma_q^t$ we have
	\begin{align}\label{31}
	\kappa_i(q,t_*) >\kappa_i(q,0) + \frac{(1+\kappa_i(q,0)^2)\tan  t_*}{1-\ep}.
	\end{align}
	Summing over $i$ in \eqref{30} and \eqref{31}, and using mean-convexity of $\Sigma$ to assert $\sum_i \kappa_i(p,0)\geq 0$ and $\sum_i \kappa_i(q,0)\geq 0$, we see that the mean curvature $H_{\Sigma^{t_*}_p}$ of $\Sigma_p^{t_*}$ at the point $x$ satisfies
	\begin{align}\label{32}
	H_{\Sigma^{t_*}_p}(x) > \frac{(n+\|A(p)\|^2)\tan t_*}{1-\ep} >0,
	\end{align}
	and likewise
	\begin{align}\label{33}
	H_{\Sigma^{t_*}_q}(x) > \frac{(n+\|A(q)\|^2)\tan t_*}{1-\ep} >0.
	\end{align}
	
	\begin{figure}
	\tikzset{every picture/.style={line width=0.75pt}} %set default line width to 0.75pt        
	
	\begin{tikzpicture}[x=0.75pt,y=0.75pt,yscale=-1,xscale=1]
	
	%Curve Lines [id:da8009984164135215] 
	\draw    (81,78) .. controls (183.88,156.05) and (420.49,166.45) .. (569.66,102.22) .. controls (588.68,94.03) and (606.28,84.63) .. (622,74) ;
	%Curve Lines [id:da471055928137011] 
	\draw    (84,208) .. controls (283,108) and (477,141) .. (624,212) ;
	%Shape: Circle [id:dp43402614684614393] 
	\draw  [fill={rgb, 255:red, 0; green, 0; blue, 0 }  ,fill opacity=1 ] (339,144.5) .. controls (339,143.12) and (340.12,142) .. (341.5,142) .. controls (342.88,142) and (344,143.12) .. (344,144.5) .. controls (344,145.88) and (342.88,147) .. (341.5,147) .. controls (340.12,147) and (339,145.88) .. (339,144.5) -- cycle ;
	%Shape: Circle [id:dp11399489737350965] 
	\draw   (275.5,144.5) .. controls (275.5,108.05) and (305.05,78.5) .. (341.5,78.5) .. controls (377.95,78.5) and (407.5,108.05) .. (407.5,144.5) .. controls (407.5,180.95) and (377.95,210.5) .. (341.5,210.5) .. controls (305.05,210.5) and (275.5,180.95) .. (275.5,144.5) -- cycle ;
	%Straight Lines [id:da5668394777116217] 
	\draw    (341.5,142) -- (341.02,100) ;
	\draw [shift={(341,98)}, rotate = 89.35] [color={rgb, 255:red, 0; green, 0; blue, 0 }  ][line width=0.75]    (10.93,-3.29) .. controls (6.95,-1.4) and (3.31,-0.3) .. (0,0) .. controls (3.31,0.3) and (6.95,1.4) .. (10.93,3.29)   ;
	%Curve Lines [id:da4763176027148388] 
	\draw  [dash pattern={on 0.84pt off 2.51pt}]  (81,62) .. controls (183.88,140.05) and (420.49,150.45) .. (569.66,86.22) .. controls (588.68,78.03) and (606.28,68.63) .. (622,58) ;
	%Curve Lines [id:da07184667373823128] 
	\draw  [dash pattern={on 0.84pt off 2.51pt}]  (82,46) .. controls (184.88,124.05) and (421.49,134.45) .. (570.66,70.22) .. controls (589.68,62.03) and (607.28,52.63) .. (623,42) ;
	%Curve Lines [id:da480175138696902] 
	\draw  [dash pattern={on 0.84pt off 2.51pt}]  (84,226) .. controls (283,126) and (477,159) .. (624,230) ;
	%Curve Lines [id:da23983639067813578] 
	\draw  [dash pattern={on 0.84pt off 2.51pt}]  (84,244) .. controls (283,144) and (477,177) .. (624,248) ;
	
	% Text Node
	\draw (335.5,149.4) node [anchor=north west][inner sep=0.75pt]  [font=\small]  {${\textstyle x}$};
	% Text Node
	\draw (55,64.4) node [anchor=north west][inner sep=0.75pt]    {$\widetilde{\Sigma}_{p}^{t_{*}}$};
	% Text Node
	\draw (57,194.4) node [anchor=north west][inner sep=0.75pt]    {$\Sigma _{q}^{t_{*}}$};
	% Text Node
	\draw (377,69.4) node [anchor=north west][inner sep=0.75pt]  [font=\small]  {$B_{r}( x)$};
	% Text Node
	\draw (343,91.0) node [anchor=north west][inner sep=0.75pt]  [font=\small]  {$\nu$};
	% Text Node
	\draw (165,62.4) node [anchor=north west][inner sep=0.75pt]  [font=\small]  {$d_{1}  >0,\ d_{2}  >0$};
	% Text Node
	\draw (108,136.4) node [anchor=north west][inner sep=0.75pt]  [font=\small]  {$d_{1} < 0,\ d_{2}  >0$};
	% Text Node
	\draw (167,218.4) node [anchor=north west][inner sep=0.75pt]  [font=\small]  {$d_{1} < 0,\ d_{2} < 0$};
	\end{tikzpicture}
	\caption{}
	\end{figure}
	
	Now, by minimality of $t_*$, the graphs $\Sigma^{t_*}_p$ and $\Sigma^{t_*}_q$ meet tangentially at $x$ and thus have opposite orientations at $x$. Let us denote by $\widetilde{\Sigma}_p^{t_*}$ the hypersurface $\Sigma_p^{t_*}$ but with the opposite orientation. Then by \eqref{32} and \eqref{33}, we have
	\begin{align}\label{51}
	H_{\widetilde{\Sigma}^{t_*}_p}(x) < 0 < H_{\Sigma^{t_*}_q}(x).  
	\end{align}

	Now denote by $d_1$ the signed distance to $\widetilde{\Sigma}_p^{t_*}$ and $d_2$ the signed distance to $\Sigma_q^{t_*}$, so that the functions $d_i$ are positive in the direction $\nu$ of the common orientation of $\widetilde{\Sigma}_p^{t_*}$ and $\Sigma_q^{t_*}$; note that the functions $d_i$ are well defined in a sufficiently small geodesic ball $B_r(x)$ (see Figure 1). Since $H_{\widetilde{\Sigma}^{t_*}_p}(x) = -\Delta d_1(x)$ and $H_{\Sigma^{t_*}_q}(x) = -\Delta d_2(x)$, \eqref{51} can be rewritten as
	\begin{align}
	-\Delta d_1(x) < 0 < -\Delta d_2(x),
	\end{align}
	and by continuity it follows that $-\Delta d_1 < 0 < -\Delta d_2$ in $B_r(x)\cap\{d_1>0\}$ for sufficiently small $r$, i.e.~$\Delta (d_2 - d_1) < 0$ in $B_r(x)\cap\{d_1>0\}$. But $d_2 - d_1 \geq 0$ in $B_r(x)\cap\{d_1>0\}$ and $(d_2 - d_1)(x) = 0$. By the Hopf lemma, either $d_2 - d_1$ is constant in $B_r(x)\cap\{d_1>0\}$, or $\nabla_\nu(d_2-d_1)(x)>0$. In either case we obtain a contradiction: the first possibility contradicts the strict inequality $\Delta (d_2 - d_1) < 0$ in $B_r(x)\cap\{d_1>0\}$, and the second possibility contradicts the fact that $\nabla_\nu d_1(x) =\nabla_\nu d_2(x) = 1$. \medskip 
	
	We have therefore shown that $t_* = T_\Sigma$. It is also clear from the computations \eqref{30}--\eqref{33} that $\Sigma^t$ is strictly mean convex for $t\in(0,T_\Sigma)$, which completes the proof of the proposition. 
\end{proof}

As a corollary of Proposition \ref{B} we obtain an explicit volume bound for closed embedded mean-convex hypersurfaces in $\mathbb{S}^{n+1}$: 

\begin{prop}\label{C}
	Suppose $\Sigma^n\subset\mathbb{S}^{n+1}$ is a smooth, closed and embedded mean-convex hypersurface with $\max_{\Sigma}\|A\|\leq \Lambda$, and define
	\begin{align}
	\operatorname{I}_\Lambda = \int_0^{\arctan(\Lambda^{-1})}(\cos t)^n(1-\Lambda \tan t)^n\,dt. 
	\end{align} 
	Then 
	\begin{align}\label{45}
	\operatorname{Vol}(\Sigma^n)\leq \frac{1}{2 \operatorname{I}_\Lambda}\operatorname{Vol}(\mathbb{S}^{n+1}).
	\end{align}
	In particular, there exists a dimensional constant $c_n \leq \frac{25}{3}\big(\frac{5}{4}\big)^{n-2}$ such that 
	\begin{align}\label{52}
	\operatorname{Vol}(\Sigma^n) \leq c_n\Lambda \operatorname{Vol}(\mathbb{S}^{n+1})
	\end{align}
	whenever $\Lambda \geq \frac{1}{4}$. 
\end{prop}

\begin{rmk}\label{201}
	Suppose in addition to the hypotheses of Proposition \ref{C} that $\Sigma$ is minimal and not totally geodesic. Then by the inequality $\int_{\Sigma}\|A\|^2(\|A\|^2 - n)\,dS_g \geq 0$ of Simons \cite{Sim68}, $\Lambda \geq \sqrt{n}$ and thus the assumption $\Lambda \geq \frac{1}{4}$ is automatically satisfied. Note that the restriction $\Lambda \geq \frac{1}{4}$ is somewhat arbitrary, allowing us to make a crude estimation of the quantity $\operatorname{I}_\Lambda$ in the proof below.
\end{rmk}

\begin{proof}
	Let $V^{\pm}(R)$ denote the volume of region swept out by the parallel hypersurfaces $\Sigma^{\pm t}$ for $0 \leq t \leq R$. Then by \cite[Exercise 3.5]{Gray04} and Proposition \ref{B}, the following formula is valid for $R\leq\arctan(\Lambda^{-1})$:
	\begin{align}
	V^{\pm}(R) = \int_\Sigma \bigg(\int_0^R (\cos t)^n \prod_{i=1}^{n}(1\mp \kappa_i\tan t)\,dt\bigg) \,dS.
	\end{align}
	Taking $R=\arctan(\Lambda^{-1})$, we therefore see that
	\begin{align}\label{49}
	\operatorname{Vol}(\mathbb{S}^{n+1})& \geq V^+(\arctan(\Lambda^{-1})) +  V^-(\arctan(\Lambda^{-1})) \nonumber \\
	& \geq 2\int_\Sigma \bigg(\int_0^{\arctan(\Lambda^{-1})} (\cos t)^n (1-\Lambda \tan t)^n\,dt\bigg)\,dS = 2\operatorname{I}_\Lambda\operatorname{Vol}(\Sigma^n),
	\end{align}
	which proves \eqref{45}. \medskip 
	
	Now suppose that $\Lambda \geq \frac{1}{4}$. Then on the interval $[0,\frac{5}{54\Lambda}]$ it is easy to verify that $\cos t \geq \frac{9}{10}$ and $\tan t \leq \frac{27t}{25} \leq \frac{1}{10\Lambda}$, and moreover $ \frac{5}{54\Lambda} \leq \arctan(\Lambda^{-1})$. Therefore, by \eqref{49} we obtain
	\begin{align}
	\operatorname{Vol}(\mathbb{S}^{n+1}) \geq 2\operatorname{Vol}(\Sigma^n) \int_0^{\frac{5}{54\Lambda}} (\cos t)^n (1-\Lambda \tan t)^n\,dt \geq \frac{5}{27}\bigg(\frac{9}{10}\bigg)^{2n}\frac{\operatorname{Vol}(\Sigma^n)}{\Lambda},
	\end{align}
	from which the estimate \eqref{52} easily follows with $c_n \leq \frac{25}{3}\big(\frac{5}{4}\big)^{n-2}$. 
\end{proof}

\section{The improved estimate}\label{s3}

In this section we prove Theorem \ref{A}. We begin in Section \ref{s31} by recalling the argument of Choi \& Wang \cite{CW83}. In Section \ref{s32} we give the proof of Theorem \ref{A} assuming the validity of two propositions. In Sections \ref{100} and \ref{101} we give the proofs of these two propositions.

\subsection{The estimate of Choi \& Wang}\label{s31}

Our proof of Theorem \ref{A} initially proceeds in the same way as in \cite{CW83}; we derive the relevant estimate of \cite{CW83} here for the convenience of the reader. The starting point is the following identity due to Reilly \cite{Rei77}, which is an integral version of Bochner's formula:

\begin{lem}[Reilly's formula]
	Let $(X^{n+1},g)$ be a smooth orientable Riemannian manifold with boundary $\Sigma^n \defeq \partial X^{n+1}$. Denote by $dv_g$ the volume element on $(X^{n+1},g)$, $dS_g$ the volume element of the induced metric on $\Sigma$,  $u_{\nu}$ the outward normal derivative of $u$ on $\Sigma$, $\nabla^\Sigma$ the gradient operator of the induced metric on $\Sigma$, $A$ the second fundamental form of $\Sigma$ defined with respect to the inward unit normal, and $H$ the mean curvature of $\Sigma$ with respect to the inward unit normal. Then for $u\in C^2(\overline{X})$, 
	\begin{align}\label{12}
	\int_X \big((\Delta u)^2 - |\nabla^2 u|^2\big)\,dv_g & = \int_X \operatorname{Ric}_X(\nabla u, \nabla u)\,dv_g + \int_{\Sigma} (\Delta^\Sigma u - Hu_{\nu})u_{\nu}\,dS_g \nonumber \\
	& \quad - \int_{\Sigma} \langle \nabla^{\Sigma} u, \nabla^{\Sigma} u_{\nu} \rangle \,dS_g - \int_{\Sigma} A(\nabla^\Sigma u, \nabla^\Sigma u)\,dS_g.
	\end{align}
\end{lem}

\begin{rmk}
	Our convention that $A$ and $H$ be defined with respect to the inward unit normal on $\Sigma$ is opposite to the convention used in \cite{CW83}.
\end{rmk}

Recall that under the setup of Theorem \ref{A}, we may write $\mathbb{S}^{n+1} = M_1 \cup M_2$, where $\partial M_1 = \partial M_2 = \Sigma$. Denote by $\Psi$ an $L^2$-normalised eigenfunction corresponding to the first non-zero eigenvalue $\lambda_1$ of $-\Delta^\Sigma$, so that $-\Delta^\Sigma \Psi = \lambda_1\Psi$ and $\|\Psi\|_{L^2(\Sigma)} =1$, and let $u$ be the unique solution to 
\begin{align}\label{6}
\begin{cases} \Delta u= 0 &   \text{in }M_1\\
u = \Psi  & \text{on }\Sigma.
\end{cases} 
\end{align}
In what follows, we fix the orientation on $\Sigma$ pointing into $M_1$, and we denote by $g$ the round metric on $\mathbb{S}^{n+1}$. We may assume that $-\int_\Sigma A(\nabla^\Sigma u, \nabla^{\Sigma} u)\,dS_g \geq 0$, otherwise we work on $M_2$ instead. Then by Reilly's formula and minimality of $\Sigma$, the solution $u$ to \eqref{6} satisfies
\begin{align}\label{3}
-\int_{M_1}  |\nabla^2 u|^2\,dv_g  & \geq n\int_{M_1} |\nabla u|^2\,dv_g + \int_{\Sigma} u_{\nu}\Delta^\Sigma u\,dS_g -  \int_{\Sigma} \langle \nabla^{\Sigma} u, \nabla^{\Sigma} u_{\nu} \rangle \,dS_g \nonumber \\
& = n\int_{M_1} |\nabla u|^2\,dv_g + 2\int_{\Sigma} u_{\nu}\Delta^\Sigma u\,dS_g \nonumber \\
& = n\int_{M_1} |\nabla u|^2\,dv_g - 2\lambda_1\int_{\Sigma} u_{\nu} u \,dS_g. 
\end{align}
On the other hand, integration by parts and the fact that $\Delta u = 0$ on $M_1$ gives
\begin{align}\label{2}
\int_{\Sigma} u_{\nu} u \,dS_g = \int_{M_1} \big(|\nabla u|^2 + u\Delta u\big)\,dv_g =  \int_{M_1}|\nabla u|^2 \,dv_g,
\end{align}
and substituting \eqref{2} back into \eqref{3} yields
\begin{align}\label{4}
2\bigg(\lambda_1 - \frac{n}{2}\bigg)\int_{M_1} |\nabla u|^2\,dv_g & \geq  \int_{M_1} |\nabla^2 u|^2\,dv_g \geq 0.  
\end{align}
This is precisely the estimate derived in \cite{CW83}; the lower bound $\lambda_1 \geq \frac{n}{2}$ follows immediately from \eqref{4}, since $|\nabla u| \not\equiv 0$. We note that in \cite{Bre13}, Brendle gave a refinement of the above argument to show that $\lambda_1>\frac{n}{2}$, although we will not need to use this strict inequality in our subsequent arguments.

\subsection{Proof of Theorem \ref{A}}\label{s32}

As seen above, the term $\int_{M_1} |\nabla^2 u|^2\,dv_g$ in \eqref{4} is simply dropped in the argument of Choi \& Wang. In order to prove Theorem \ref{A}, we obtain a lower bound for $\int_{M_1} |\nabla^2 u|^2\,dv_g$ in terms of $\int_{M_1}|\nabla u|^2\,dv_g$. \medskip

Our proof of Theorem \ref{A} can be decomposed into two main propositions, which we describe now. We introduce parameters $0<\ep\leq \frac{\Lambda}{2}$ and $\beta>0$, which are to be fixed later in the proof of Theorem \ref{A} but assumed sufficiently small for now so that $\gamma \defeq \sqrt{2n} - \frac{\Lambda \ep}{\Lambda-\ep}(\frac{n}{\Lambda^2}+1)-\beta>0$. We also define $\delta = n\arctan(\frac{\ep}{n})$ and $T = \frac{\delta}{2\Lambda^2}$, and for $t\geq 0$ we denote $M_1^t = \{x\in M_1: d(x)>t\}$, where $d$ is the distance to $\Sigma$ in $M_1$. Note $\partial M_1^t = \Sigma^t$ is a smooth embedded hypersurface for $0 \leq t<\arctan(\Lambda^{-1})$ by Proposition \ref{B}, and in particular this holds for $0\leq t< 2T$.  Our two main propositions are then as follows: 

\begin{prop}\label{p1}
	Suppose $\ep, \beta,\gamma, \delta$ and $T$ are as above, and $\Lambda \geq \sqrt{n}$. Then
	\begin{align}\label{26}
	\int_{M_1}|\nabla u|^2\,dv_g \leq \frac{2\Lambda^2}{\delta\gamma}\int_{M_1^{T}\backslash M_1^{2T}} |\nabla u|^2\,dv_g + \frac{1}{\beta\gamma}\int_{M_1}|\nabla^2u|^2\,dv_g. 
	\end{align}
\end{prop}
\begin{prop}\label{p2}
	Suppose $\ep, \beta, \gamma, \delta$ and $T$ are as above. Then 
	\begin{align}\label{47}
	\int_{M_1^T} |\nabla u|^2\,dv_g \leq \frac{16\Lambda^4}{(n-1)\delta^2}\int_{M_1}|\nabla^2 u|^2\,dv_g.
	\end{align}
\end{prop}

Assuming the validity of Propositions \ref{p1} and \ref{p2} for now, we proceed to give the proof of Theorem \ref{A}: 

\begin{proof}[Proof of Theorem \ref{A}]
	 By Simons' inequality $\int_{\Sigma}\|A\|^2(\|A\|^2 - n)\,dS_g \geq 0$ for minimal hypersurfaces in $\mathbb{S}^{n+1}$ \cite{Sim68}, if $\Lambda < \sqrt{n}$ then $A \equiv 0$ and thus $\Sigma$ is a totally geodesic $n$-sphere. In this case, it is well-known that $\lambda_1(\Sigma) = n$, and so \eqref{-1} clearly holds. For the remainder of the proof, we may therefore assume that $\Lambda \geq \sqrt{n}$. \medskip 
	
	Substituting the estimate \eqref{47} of Proposition \ref{p2} back into the estimate \eqref{26} of Proposition \ref{p1}, we obtain 
	\begin{align}
	\int_{M_1}|\nabla u|^2\,dv_g \leq \bigg(\frac{32\Lambda^6}{(n-1)\delta^3 \gamma} + \frac{1}{\beta\gamma}\bigg)\int_{M_1}|\nabla^2 u|^2\,dv_g.
	\end{align}
	Therefore
	\begin{align}\label{40}
	\int_{M_1}|\nabla^2 u|^2\,dv_g \geq \frac{a_n}{\Lambda^6 + b_n}\int_{M_1}|\nabla u|^2\,dv_g
	\end{align}
	where
	\begin{align}\label{41}
	a_n = \frac{(n-1)\delta^3\gamma}{32}  \quad \text{and} \quad b_n = \frac{(n-1)\delta^3}{32\beta}.
	\end{align}
	Now, since we assume $\ep\leq \frac{\Lambda}{2}$ we have $\frac{\Lambda}{\Lambda - \ep} \leq 2$, and since we assume $\Lambda \geq \sqrt{n}$ we have $\frac{n}{\Lambda^2} \leq 1$. Substituting these inequalities back into the definition of $\gamma$, we see
	\begin{align}
		\gamma & = \sqrt{2n} -\frac{\Lambda\ep}{\Lambda-\ep}\bigg(\frac{n}{\Lambda^2}+1\bigg) - \beta \geq \sqrt{n}\bigg(\sqrt{2} - \frac{4\ep}{\sqrt{n}}- \frac{\beta}{\sqrt{n}}\bigg). 
	\end{align}
	Choosing $\beta=\frac{\sqrt{n}}{20}$ and $\ep = \frac{\sqrt{n}}{3}$, we then see that $\gamma  \geq \sqrt{n}(\sqrt{2} - \frac{4}{3} - \frac{1}{20}) \geq \frac{3\sqrt{n}}{100}$ and $\delta = n\arctan(\frac{\ep}{n}) = n\arctan(\frac{1}{3\sqrt{n}})$. Therefore
	\begin{align}
	a_n \geq  \frac{3(n-1)n^{7/2}}{3200}\arctan^3\bigg(\frac{1}{3\sqrt{n}}\bigg) \quad \text{and} \quad b_n \leq \frac{5n^{7/2}}{8}\arctan^3\bigg(\frac{1}{3\sqrt{n}}\bigg).
	\end{align}
	As explained in Remark \ref{53}, this completes the proof of the theorem.
\end{proof}

The rest of the paper is devoted to the proofs of Propositions \ref{p1} and \ref{p2}. 

\subsection{Proof of Proposition \ref{p1}}\label{100}

To describe our setup for the proof of Proposition \ref{p1}, let $d$ be the signed distance to $\Sigma$ in $\mathbb{S}^{n+1}$:
\begin{align}
d(x) = \begin{cases}
-\operatorname{dist}(x,\Sigma) & \text{if }x\in \overline{M_2} \\
\operatorname{dist}(x,\Sigma) & \text{if }x\in M_1. 
\end{cases}
\end{align} 
As before, we equip the surfaces $\Sigma^d$ with the orientation induced by $\Sigma$, i.e.~the orientation given by the normal vector field $\nabla d$ on $\Sigma^d$. Then the mean curvature of $\Sigma^d$ is given by $H_{\Sigma^d} = -\operatorname{div}\nabla d = -\Delta d$. \medskip 

By Proposition \ref{B}, the parallel hypersurfaces $\Sigma^d$ are smooth and embedded for 
\begin{align}\label{46}
|d| \in \big[0,\arctan(\Lambda^{-1})\big).
\end{align}
However, to gain control on the mean curvature of the hypersurfaces parallel to $\Sigma$, in the proof of Proposition \ref{p1} we will need to work in a neighbourhood around $\Sigma$ of thickness smaller than that determined by \eqref{46}. To this end, for $0<\ep\leq \frac{\Lambda}{2}$ we define 
\begin{align}
D_\ep = \arctan(\ep\Lambda^{-2}).
\end{align}
Since $\frac{\ep}{\Lambda}<1$, clearly $D_\ep < \arctan(\Lambda^{-1})$ and thus $\Sigma^t$ is a smooth embedded hypersurface for $|t|\in[0,D_\ep]$. Our first estimate towards the proof of Proposition \ref{p1} is an upper bound on the mean curvature of the hypersurfaces $\Sigma^t$ parallel to $\Sigma$ when $t\in[0,D_\ep]$:

\begin{lem}\label{43}
	Let $0<\ep\leq \frac{\Lambda}{2}$. Then for $t\in[0,D_\ep]$,
	\begin{align}
	H_{\Sigma^t} \leq \widetilde{\ep} \defeq \frac{\Lambda \ep}{\Lambda -\ep}\bigg(\frac{n}{\Lambda^2}+1\bigg).
	\end{align}
\end{lem}

\begin{proof}
	Summing over $i$ in \eqref{36} and appealing to minimality of $\Sigma$, we see that for $t\in[0,D_\ep]$ we have
	\begin{align}\label{27}
	H_{\Sigma^t} = \sum_{i=1}^n \bigg( \kappa_i(\cdot,0) +  \frac{(1+\kappa_i(\cdot,0)^2)\tan t}{1-\kappa_i(\cdot,0)\tan t}\bigg) = \sum_{i=1}^n\frac{(1+\kappa_i(\cdot,0)^2)\tan t}{1-\kappa_i(\cdot,0)\tan t}.
	\end{align}
	Now, by definition of $D_\ep$, we have $1 - \kappa_i(\cdot,0)\tan t \geq \frac{\Lambda - \ep}{\Lambda}$ on $[0,D_\ep]$ for each $i$, and it therefore follows from \eqref{27} that for $t\in[0,D_\ep]$,
	\begin{align}\label{28}
	H_{\Sigma^t}   \leq \frac{\Lambda}{\Lambda -\ep}(n+\Lambda^2)\tan t  &= \frac{\Lambda}{\Lambda-\ep}\bigg(\frac{ n}{\Lambda^2}+1\bigg)\Lambda^2 \tan t \leq \frac{\Lambda \ep}{\Lambda -\ep}\bigg(\frac{n}{\Lambda^2}+1\bigg),
	\end{align}
	as claimed.
\end{proof}

We now use Lemma \ref{43} to show: 

\begin{lem}\label{48}
	Let $0<\ep\leq \frac{\Lambda}{2}$ and suppose $v$ is a smooth function defined on $\overline{M_1}$. Then for $t\in[0,D_\ep]$ and any $\beta>0$, 
	\begin{align}\label{18}
	\int_{\Sigma} |\nabla v|^2\,dS_g \leq \int_{\Sigma^t} |\nabla v|^2\,dS_g + (\widetilde{\epsilon}+\beta)\int_{M_1\backslash M_1^t}|\nabla v|^2\,dv_g + \beta^{-1}\int_{M_1\backslash M_1^t}|\nabla^2 v|^2\,dv_g. 
	\end{align}
\end{lem}
\begin{proof}
	Recall that if $x\in \mathbb{S}^{n+1}$ is a signed distance $s$ from $\Sigma$, then $H_{\Sigma^s}(x) = -\Delta d(x)$. By \eqref{28}, we therefore have
	\begin{align}\label{17}
	-\int_{M_1\backslash M_1^t} |\nabla v|^2\Delta d\,dv_g  \leq \widetilde{\epsilon}\int_{M_1\backslash M_1^t}|\nabla v|^2\,dv_g. 
	\end{align} 
	On the other hand, by the divergence theorem
	\begin{align}\label{44}
	-\int_{M_1\backslash M_1^t} |\nabla v|^2\Delta d\,dv_g & = \int_{M_1\backslash M_1^t}\langle \nabla d, \nabla|\nabla v|^2\rangle \,dv_g - \int_{\Sigma\cup\Sigma^t}|\nabla v|^2\langle \nabla d, \nu\rangle \,dS_g,
	\end{align} 
	where $\nu$ is the outward pointing unit normal to the region $M_1\backslash M_1^t$. By definition of $d$, we have $\langle \nabla d, \nu\rangle = -1$ on $\Sigma$ and $\langle \nabla d, \nu\rangle = 1$ on $\Sigma^t$. Therefore, by \eqref{44}, for any $\beta>0$ we have
	\begin{align}\label{16}
	-\int_{M_1\backslash M_1^t} |\nabla v|^2\Delta d\,dv_g & \geq -2\int_{M_1\backslash M_1^t}|\nabla v||\nabla^2 v|\,dv_g + \int_{\Sigma}|\nabla v|^2\,dS_g - \int_{\Sigma^t}|\nabla v|^2\,dS_g \nonumber \\
	& \geq -\beta\int_{M_1\backslash M_1^t}|\nabla v|^2\,dv_g - \beta^{-1}\int_{M_1\backslash M_1^t}|\nabla^2 v|^2\,dv_g + \int_{\Sigma}|\nabla v|^2\,dS_g \nonumber \\
	& \qquad - \int_{\Sigma^t}|\nabla v|^2\,dS_g.
	\end{align}
	Substituting \eqref{16} into \eqref{17} and rearranging, we arrive at \eqref{18}. 
\end{proof}

Whilst the desired estimate in Proposition \ref{p1} involves $\int_{M_1}|\nabla u|^2\,dv_g$ on the LHS, the estimate in Lemma \ref{48} (therein taking $v=u$) involves $\int_{\Sigma}|\nabla u|^2\,dS_g$ on the LHS. These two quantities are related by the following lemma: 

\begin{lem}\label{21}
	The solution $u$ to \eqref{6} satisfies
	\begin{align}\label{20}
	\int_{\Sigma} |\nabla u|^2\,dS_g \geq \sqrt{2n}\int_{M_1}|\nabla u|^2\,dv_g. 
	\end{align}
\end{lem}
\begin{proof}
	Integrating by parts, using $\Delta u = 0$ in $M_1$ and the fact that $\|u|_{\Sigma}\|_{L^2(\Sigma)}=\|\Psi\|_{L^2(\Sigma)}=1$, we have
	\begin{align}\label{10}
	\bigg(\int_{M_1} |\nabla u|^2\,dv_g\bigg)^2 = \bigg(\int_{\Sigma} u_{\nu} u\,dS_g\bigg)^2 \leq\int_{\Sigma} u_{\nu}^2\,dS_g\int_{\Sigma} u^2\,dS_g = \int_{\Sigma} u_{\nu}^2\,dS_g. 
	\end{align}
	On the other hand, 
	\begin{align}\label{11}
	\int_{\Sigma} u_{\nu}^2\,dS_g = \int_{\Sigma} |\nabla u|^2\,dS_g - \int_{\Sigma}|\nabla^\Sigma u|^2\,dS_g = \int_{\Sigma} |\nabla u|^2\,dS_g -\lambda_1, 
	\end{align}
	with the second identity in \eqref{11} following from the variational characterisation of $\lambda_1$ and the fact that $u|_{\Sigma} = \Psi$. Substituting \eqref{11} into \eqref{10} and applying Young's inequality, we obtain
	\begin{align}\label{42}
	\int_{\Sigma} |\nabla u|^2\,dS_g \geq \lambda_1 + 	\bigg(\int_{M_1} |\nabla u|^2\,dv_g\bigg)^2 \geq 2\lambda_1^{1/2}\int_{M_1}|\nabla u|^2\,dv_g. 
	\end{align}
	The desired estimate \eqref{20} then follows from \eqref{42} and the fact that $\lambda_1 \geq \frac{n}{2}$. 
\end{proof}

We are now in a position to give the proof of Proposition \ref{p1}:

\begin{proof}[Proof of Proposition \ref{p1}]
	We first take $v=u$ in the estimate \eqref{18} of Lemma \ref{48}, where $u$ is the solution to \eqref{6}. Substituting \eqref{20} back into \eqref{18}, we therefore arrive at
	\begin{align}\label{23}
	\underbrace{(\sqrt{2n} - \widetilde{\epsilon} - \beta)}_{=:\gamma}\int_{M_1}|\nabla u|^2\,dv_g \leq \int_{\Sigma^t} |\nabla u|^2\,dS_g + \beta^{-1} \int_{M_1}|\nabla^2 u|^2\,dv_g. 
	\end{align}
	Now recall that we define $\delta = n \arctan(\frac{\ep}{n})$. Noting that $\frac{\delta}{x^2} \leq \arctan(\frac{\ep}{x^2})$ for $x \geq \sqrt{n}$, we see that $\frac{\delta}{\Lambda^2} \leq \arctan(\frac{\ep}{\Lambda^2}) = D_\ep$ for $\Lambda \geq \sqrt{n}$. In particular, we are justified in integrating both sides of \eqref{23} with respect to $t$ over the interval $[T,2T]$, where $T=\frac{\delta}{2\Lambda^2}$. This yields \eqref{26}, completing the proof of Proposition \ref{p1}. 
\end{proof}

\subsection{Proof of Proposition \ref{p2}}\label{101}

The proof of Proposition \ref{p2} is a consequence of two lemmas. The first of these is as follows: 

\begin{lem}\label{14}
	Let $\Omega\subset \mathbb{S}^{n+1}$ be a domain and $v$ a smooth function defined on $\Omega$ satisfying $\Delta v = 0$ in $\Omega$. Then
	\begin{align}\label{13}
	\Delta|\nabla v|^2 = 2|\nabla^2 v|^2 + 2n|\nabla v|^2 \quad \text{in }\Omega. 
	\end{align}
\end{lem}
\begin{proof}
	This is an immediate consequence of the Bochner formula 
		\begin{align}
	\Delta|\nabla w|^2 = 2\langle \nabla \Delta w, \nabla w\rangle + 2|\nabla^2 w |^2 + 2\operatorname{Ric}_g(\nabla w,\nabla w)
	\end{align}
	for a smooth function $w$ defined on a Riemannian manifold $(N,g)$, and the fact that $\operatorname{Ric}_g = ng$ on $\mathbb{S}^{n+1}$ equipped with the round metric $g$.
\end{proof}

For a domain $\Omega\subset \mathbb{S}^{n+1}$ with smooth boundary, we denote by $\Omega^{s}$ the set of points in $\Omega$ whose distance to $\partial \Omega$ is greater than $s$.  We now use Lemma \ref{14} to show:

\begin{lem}\label{24}
	Let $\Omega\subset\mathbb{S}^{n+1}$ be a domain with smooth boundary $\partial \Omega$, and $v$ a smooth function defined on $\Omega$ satisfying $\Delta v = 0$ in $\Omega$. Suppose that $t>0$ is sufficiently small so that $\partial(\Omega^{2t})$ is a smooth embedded hypersurface in $\mathbb{S}^{n+1}$. Then
	\begin{align}\label{15}
	\int_{\Omega^{2t}} |\nabla v|^2\,dv_g \leq \frac{1}{n-1}t^{-2} \int_{\Omega} |\nabla^2 v|^2\,dv_g.
	\end{align}
\end{lem}
\begin{proof}
	Let $\zeta\in C_c^\infty(\Omega)$ be a cutoff function whose properties will be specified later in the proof. Multiplying the inequality \eqref{13} by $\zeta^2$ and integrating over $\Omega$, we see 
	\begin{align}
	\int_{\Omega} \zeta^2(n|\nabla v|^2 + |\nabla^2 v|^2)\,dv_g & = \frac{1}{2}\int_{\Omega} \zeta^2 \Delta |\nabla v|^2\,dv_g \nonumber \\
	& = -\int_{\Omega} \zeta \langle \nabla \zeta, \nabla |\nabla v|^2\rangle \,dv_g \nonumber \\
	& = - 2\int_{\Omega} \zeta \nabla^2 v(\nabla v,\nabla\zeta)\,dv_g \nonumber \\
	& \leq  \int_{\Omega} \zeta^2 |\nabla v|^2\,dv_g +\int_{\Omega} |\nabla\zeta|^2|\nabla^2 v|^2\,dv_g.
	\end{align}
	Therefore
	\begin{align}
	0 &  \geq (n-1)\int_{\Omega} \zeta^2 |\nabla v|^2\,dv_g + \int_{\Omega}(\zeta^2 - |\nabla\zeta|^2)|\nabla^2 v|^2\,dv_g \nonumber \\
	& \geq (n-1)\int_{\Omega} \zeta^2 |\nabla v|^2\,dv_g - \int_{\Omega} |\nabla\zeta|^2|\nabla^2 v|^2\,dv_g,
	\end{align}
	which yields
	\begin{align}\label{54}
	\int_{\Omega} \zeta^2 |\nabla v|^2\,dv_g \leq \frac{1}{n-1}\int_{\Omega} |\nabla\zeta|^2|\nabla^2 v|^2\,dv_g.
	\end{align}
	
	Now, for each $\ep>0$, one can choose a smooth cutoff function $\zeta$ such that $\zeta \equiv 0$ in $\Omega\backslash \Omega^t$, $\zeta \equiv 1$ in $\Omega^{2t}$ and $|\nabla \zeta| \leq (1+\ep)t^{-1}$. Thus \eqref{54} implies 
	\begin{align}\label{55}
	\int_{\Omega^{2t}}|\nabla v|^2\,dv_g \leq \frac{(1+\ep)^2}{n-1}t^{-2}\int_{\Omega}|\nabla^2 v|^2\,dv_g,
	\end{align}
	and the estimate \eqref{15} then follows after taking $\ep\rightarrow 0$ in \eqref{55}. 
\end{proof}

\begin{proof}[Proof of Proposition \ref{p2}]
In the statement of Lemma \ref{24}, let $\Omega = M_1$ and let $v=u$, where $u$ is the solution to \eqref{6}. Following the reasoning given in the proof Proposition \ref{p1}, we are then justified in taking $t = T/2$ in Lemma \ref{24}, where $T= \frac{\delta}{2\Lambda^2}$ as before. The desired estimate \eqref{47} then follows. 
\end{proof}

Having established Propositions \ref{p1} and \ref{p2}, the proof of Theorem \ref{A} is complete, as explained in Section \ref{s32}. 

\section*{Acknowledgements} YS is partially supported by NSF DMS grant $2154219$, ``Regularity {\sl vs} singularity formation in elliptic and parabolic equations".

\vspace*{-4mm}
\bibliography{references}{}
\bibliographystyle{siam}

	\end{document}